\documentclass[letterpaper, 10pt]{amsart}

\usepackage{amssymb}
\usepackage{color}
\usepackage{graphicx}
\usepackage{ifthen}
\usepackage{mathrsfs} 
\usepackage{mathscinet} 
\usepackage{mathtools}
\usepackage{rotating}
\usepackage{stmaryrd}
\usepackage{linsys}
\usepackage{braket}
\usepackage{hyperref}
\usepackage{enumitem} 

\newcommand*{\Z}{\mathbb{Z}}
\newcommand*{\Q}{\mathbb{Q}}
\newcommand*{\R}{\mathbb{R}}

\DeclarePairedDelimiter{\angles}{\langle}{\rangle}
\DeclarePairedDelimiter{\floors}{\lfloor}{\rfloor}

\newcommand*{\verts}[1]{\left\lvert #1 \right\rvert}
\newcommand*{\braces}[1]{\left\lbrace #1 \right\rbrace}
\newcommand*{\parens}[1]{\left\lparen #1 \right\rparen}

\newcommand*{\card}{\verts}
\newcommand*{\setof}{\braces}

\newcommand*{\inner}{\angles}

\newcommand*{\deftobe}{\mathrel{\coloneqq}}

\newcommand*{\maps}{\colon}
\newcommand*{\st}{\,:\,}

\renewcommand*{\epsilon}{\varepsilon}
\renewcommand*{\phi}{\varphi}
\renewcommand{\subset}{\subseteq}

\newtheorem{thm}{Theorem}[section]

\newtheorem{lem}[thm]{Lemma}

\newtheorem{prop}[thm]{Proposition}

\theoremstyle{definition}

\newtheorem{defn}[thm]{Definition}

\theoremstyle{remark}

\usepackage{subfig}
\newcommand*{\SL}{\mathsf{SL}}
\newcommand*{\GL}{\mathsf{GL}}

\DeclareMathOperator*{\latlen}{len}  
\newcommand*{\bnd}{\mathfrak{b}}  
\newcommand*{\intr}{\mathfrak{I}} 
\newcommand*{\A}{\mathfrak{A}}    
\newcommand*{\inhl}[1]{\widetilde{#1}}   
\newcommand*{\tran}{\mathrm{t}}  
\renewcommand*{\card}[1]{\verts{#1}}
\newcommand{\ind}[1]{j_{#1}}
\newcommand{\pZ}{\mathrm{p}\Z}

\DeclareMathOperator{\conv}{Conv}

\newcommand*{\ehr}{\mathfrak{L}}
\DeclareMathOperator{\lcm}{lcm}
\newcommand{\dual}[1]{#1^{\vee}}
\DeclareMathOperator{\den}{den}

\theoremstyle{plain}
\newtheorem*{thmCoefficientPeriodSequences}{%
   Theorem \ref{thm:CoefficientPeriodSequences}%
}
\newtheorem*{thmBelowScottsRegion}{%
   Theorem \ref{thm:BelowScottsRegion}%
}

\title{%
   Ehrhart quasi-period collapse in rational polygons
}%

\author[T.~B.~McAllister]{%
   Tyrrell B. McAllister
}
\address[T.~B.~McAllister]{%
   Department of Mathematics \\
   University of Wyoming \\
   Laramie, WY 82071 \\
   USA
}
\email{%
   tmcallis@uwyo.edu
}
\thanks{%
   First author supported by the Netherlands Organisation for
   Scientific Research (NWO) Mathematics Cluster DIAMANT. Second
   author supported by Wyoming EPSCoR%
}%

\author{%
   Matthew Moriarity
}



\begin{document}

\begin{abstract}
   In 1976, P.~R.~Scott characterized the Ehrhart polynomials
   of convex integral polygons.  We study the same question for
   Ehrhart polynomials and quasi-polynomials of
   \emph{non}-integral convex polygons.
%
   Turning to the case in which the Ehrhart quasi-polynomial has
   nontrivial quasi-period, we determine the possible minimal
   periods of the coefficient functions of the Ehrhart
   quasi-polynomial of a rational polygon.
\end{abstract}

\maketitle

\section{%
   Introduction%
}%

A \emph{rational polygon} $P \subset \R^{2}$ is the convex hull of
finitely many rational points, not all contained in a line.  Given
a positive integer $n$, let $nP \deftobe \{nx \in \R^{2} \st x \in
P\}$ be the dilation of $P$ by $n$.  The $2$-dimensional case of a
well-known result due to Ehrhart \cite{Ehr1962b} states that the
lattice-point enumerator $n \mapsto \card{nP \cap \Z^{2}}$ for $P$
is a degree-$2$ quasi-polynomial function with rational
coefficients.  That is, there exist periodic functions $c_{P,0},
c_{P,1}, c_{P,2} \maps \Z \to \Q$ with $c_{P,2} \not\equiv 0$ such
that, for all positive integers $n$,
\begin{align*}
   \ehr_{P}(n)
   &\deftobe
      c_{P,0}(n)
      +  c_{P,1}(n) n
      + c_{P,2}(n) n^{2}
   = \card{nP \cap \Z^{2}}.
\end{align*}
We call $\ehr_{P}$ the \emph{Ehrhart quasi-polynomial} of $P$.
The \emph{period sequence} of $P$ is $(s_{0}, s_{1}, s_{2})$,
where $s_{i}$ is the minimum period of the coefficient function
$c_{P,i}$ for $i = 0, 1, 2$.  The \emph{quasi-period} of
$\ehr_{P}$ (or of $P$) is $\lcm \setof{s_{0}, s_{1}, s_{2}}$.  We
refer the reader to \cite{BecRob2007} and \cite[Chapter
4]{Sta1997} for thorough introductions to the theory of Ehrhart
quasi-polynomials.

Our goal is to examine the possible periods and values of the
coefficient functions $c_{P,i}$.  The leading coefficient function
$c_{P,2}$ is always a constant equal to the area $\A_{P}$ of $P$.
Furthermore, when $P$ is an \emph{integral} polygon (meaning that
its vertices are all in $\Z^{2}$), $\ehr_{P}$ is simply a
polynomial with $c_{P,0}=1$ and $c_{P,1} = \frac{1}{2} \bnd_{P}$,
where $\bnd_{P}$ is the number of lattice points on the boundary
of $P$.  When $P$ is integral, Pick's formula $\A_{P} = \intr_{P}
+ \frac{1}{2}\bnd_{P} - 1$ determines $\A_{P}$ in terms of
$\bnd_{P}$ and the number $\intr_{P}$ of points in the interior of
$P$ \cite{Pic1899}.  Hence, characterizing the Ehrhart polynomials
of integral polygons amounts to determining the possible numbers
of lattice points in their interiors and on their boundaries.
This was accomplished by P.~R.~Scott in 1976:

\begin{thm}[Scott \cite{Sco1976}%
   ; see also \cite{HaaSch2009}%
   ]
   \label{thm:ScottsTheorem}
   Given non-negative integers $I$ and $b$, there exists an
   integral polygon $P$ such that $(\intr_{P}, \bnd_{P}) = (I, b)$
   if and only if $b \ge 3$ and either $I = 0$, $(I, b) = (1, 9)
   $, or $b \le 2 I + 6$.
\end{thm}

In Figure~\ref{fig:ScottsRegion}, the small squares indicate the
values of $I$ and $b$ that are realized as the number of interior
lattice points and boundary lattice points of some convex integral
polygon.  After a suitable linear transformation using Pick's
Formula, these squares represent all of the Ehrhart polynomials of
integral polygons.

However, not all Ehrhart polynomials of polygons come from
\emph{integral} polygon.  Indeed, the complete characterization of
Ehrhart polynomials of rational polygons, including the
non-integral ones, remains open.  To this end, we define a
\emph{polygonal pseudo-integral polytope}, or \emph{polygonal
PIP}, to be a rational polygon with quasi-period equal to~$1$.
That is, polygonal PIPs are those polygons that share with
integral polygons the property of having a polynomial Ehrhart
quasi-polynomial.  Like integral polygons, polygonal PIPs must
satisfy Pick's Theorem \cite[Theorem~3.1]{MW2005}, so, again, the
problem reduces to finding the possible values of $\intr_{P}$ and
$\bnd_{P}$.  In Section \ref{sec:NonintegralPIPs}, we construct
polygonal PIPs with $\bnd_{P} \in \setof{1,2}$ and $\intr_{P} \ge
1$ arbitrary.  This construction therefore yields infinite
families of Ehrhart polynomials that are not the Ehrhart
polynomials of any integral polytope.  This is our first main
result, which we prove in Section 3.
\begin{thm}[proved on p.~\pageref{proof:BelowScottsRegion}]
   \label{thm:BelowScottsRegion}
   Given integers $I \ge 1$ and $b \in \{1, 2\}$, there exists a
   polygonal PIP $P$ with $(\intr_{P}, \bnd_{P}) = (I, b)$.
   However, there does not exist a polygonal PIP $P$ with
   $\bnd_{P} = 0$ or with $(\intr_{P}, \bnd_{P}) \in \{ (0,1), (0,
   2) \}$.
\end{thm}
\begin{figure}
   \subfloat[]{
      \includegraphics{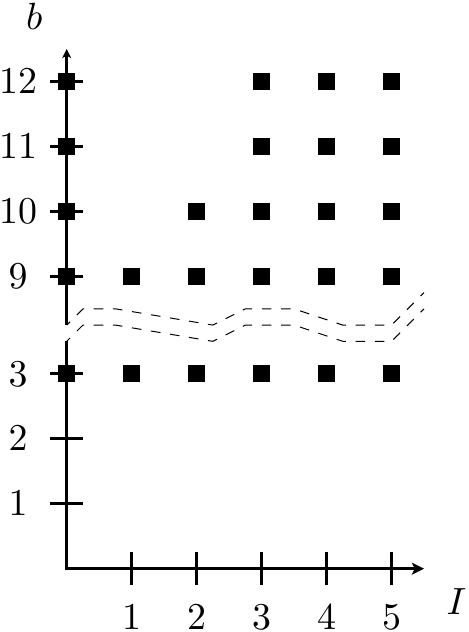}
   } \hfill
   \subfloat[]{
      \includegraphics{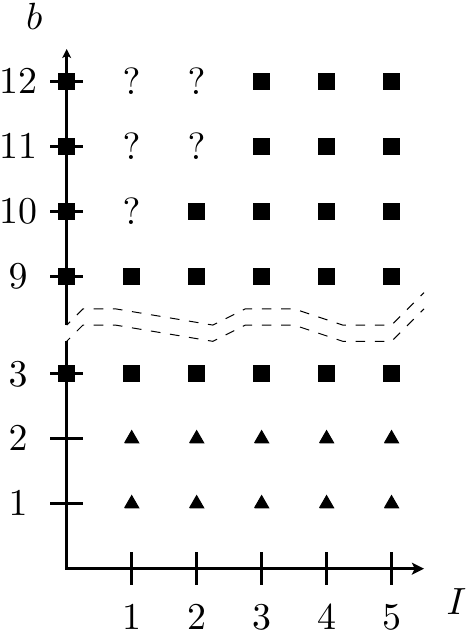}
      \label{fig:ScottRegionAll}
   }
   \caption{%
   On the left: Small squares indicate values of $(I,b)$
   corresponding to convex integral polygons
   (Theorem~\ref{thm:ScottsTheorem}).  On the right: Small
   triangles indicate additional values of $(I,b)$ corresponding
   to nonintegral PIPs (Theorem~\ref{thm:BelowScottsRegion}).
   Question marks indicate values for which the existence of
   corresponding PIPs remains open.
%
   }
   \label{fig:ScottsRegion}
\end{figure}

In Figure~\ref{fig:ScottRegionAll}, the small triangles below the
small squares indicate the values of $I$ and $b$ mentioned in
Theorem~\ref{thm:BelowScottsRegion}.  The question of whether any
points $(I, b)$, $I \ge 1$, above the small squares are realized
by PIPs remains open.

In Section \ref{sec:PeriodsOfCoefficients}, we consider rational
polygons $P$ that are not PIPs\@.  Determining all possible
coefficient functions $c_{P,i}$ seems out of reach at this time.
However, one interesting question we can answer is, What are the
possible period sequences $(s_{0}, s_{1}, s_{2})$?  In
\cite{McM1978}, P.~McMullen gave bounds on the $s_{i}$ in terms of
the \emph{indices} of~$P$.  Given a $d$-dimensional polytope $P$
and $i \in \setof {0, \dotsc, d}$, the \emph{$i$-index} of $P$ is
the least positive integer $\ind{i}$ such that every
$i$-dimensional face of the dilate $\ind{i} P$ contains an integer
lattice point in its affine span.  We call $(\ind{0}, \dotsc,
\ind{d})$ the \emph{index sequence} of $P$.  We state McMullen's
result in the general case of $d$-dimensional polytopes:

\begin{thm}[McMullen \protect{\cite[Theorem 6]{McM1978}}]
\label{thm:McMullenBound}
   Let $P$ be a $d$-dimensional rational polytope with period
   sequence $(s_{0}, \dotsc, s_{d})$ and index sequence $(\ind{0},
   \dotsc, \ind{d})$.  Then $s_{i}$ divides $\ind{i}$ for $0 \le i
   \le d$.  In particular, $s_{i} \le \ind{i}$.
\end{thm}
We will refer to the inequalities $s_{i} \le \ind{i}$ in Theorem
\ref{thm:McMullenBound} as \emph{McMullen's bounds}.  It is easy
to see that the indices $j_{i}$ of a rational polytope satisfy the
divisibility relations $\ind{d} \mid \ind{d-1} \mid \dotsb \mid
\ind{0}$, and hence $\ind{d} \le \ind{d-1} \le \dotsb \le
\ind{0}$.  Beck, Sam, and Woods \cite{BSW2008} showed that
McMullen's bounds are always tight in the $i \in \setof{d-1, d}$
cases.  It is also shown in \cite{BSW2008} that, given any
positive integers $q_{d} \mid q_{d-1} \mid \dotsb \mid q_{0}$,
there exists a polytope with $i$-index $q_{i}$ for $0 \le i \le
d$.  Moreover, all of McMullen's bounds are tight for this
polytope.

Seeing this, one might hope that the coefficient-periods $s_{i}$
in the period sequence are also required to satisfy some
constraints.  However, our second main result shows that, in the
case of polygons, $s_{0}$ and $s_{1}$ may take on arbitrary
values.  (Of course, we always have $s_{2} = 1$, because the
leading coefficient function $c_{P,2}$ is a constant.)
\begin{thm}[proved on
p.~\pageref{proof:CoefficientPeriodSequences}]
\label{thm:CoefficientPeriodSequences}
   Given positive integers $r$ and $s$, there exists a polygon
   $P$ with period sequence $(r, s, 1)$.
\end{thm}
%
Thus, in constrast to the Beck--Sam--Woods construction, the
McMullen bound $s_{0} \le \ind{0}$ can be arbitrarily far from
tight.

We prove Theorem \ref{thm:CoefficientPeriodSequences} in Section
\ref{sec:PeriodsOfCoefficients}.  Before giving proofs of Theorems
\ref{thm:BelowScottsRegion}
and~\ref{thm:CoefficientPeriodSequences}, we define in
Section~\ref{sec:PiecewiseSkewUnimodularTransformations} some
notation and terminology that we will use in our constructions.

\section{%
   Piecewise skew unimodular transformations
}%
\label{sec:PiecewiseSkewUnimodularTransformations}

Since we will be exploring the possible Ehrhart quasi-polynomials
of polygons, it will be useful to have geometric tools for
constructing rational polygons while controlling their Ehrhart
quasi-polynomials.  The main tool that we will use are piecewise
affine unimodular transformations.  Following~\cite{Gre1993}, we
call these \emph{$\pZ$-morphisms}.
\begin{defn}
   Given $U, V \subset \R^{2}$ and a finite set $\setof{\ell_{i}}$
   of lines in the plane, let $\mathcal{C}$ be the set of
   connected components of $U \setminus \bigcup_{i} \ell_{i}$.  An
   injective continuous map $f\maps U \to V$ is a
   \emph{$\pZ$-morphism} if, for each component $C \in
   \mathcal{C}$, $f \vert_{C}$ is the restriction to $C$ of an
   affine transformation, and $f\vert_{C \cap \Z^{2}}$ is the
   restiction to $C \cap \Z^{2}$ of an affine automorphism of the
   lattice.  (That is, $f\vert_{C}$ can be extended to an element
   of $\GL_{2}(\Z) \ltimes \Z^{2}$).
\end{defn}
Thus, $\pZ$-morphisms are piecewise affine linear maps that map
lattice points, and only lattice points, to lattice points.  The
key property of $\pZ$-morphisms is that they preserve the lattice
and so preserve Ehrhart quasi-polynomials.

We will only need $\pZ$-morphisms that act as skew transformations
on each component of their domains.  It will be convenient to
introduce some notation for these transformations.  Given a
rational vector $r \in \Q^{2}$, let $r_{p}$ be the generator of
the semigroup $(\R_{\ge 0}r) \cap \Z^{2}$, and define the
\emph{lattice length} $\latlen(r)$ of $r$ by $\latlen(r) r_{p} =
r$.  Thus, if $r = (\tfrac{a}{b}, \tfrac{c}{d})$, where the
fractions are reduced, we have that $\latlen(r) = \gcd(a, c) /
\lcm(b, d)$.  Define the skew unimodular transformation $U_{r} \in
\SL_{2}(\Z)$ by
\begin{equation*}
   U_{r} (x) 
   =  x 
      +  \frac{1}{\latlen(r)^{2}}
         \det(r, x) r,
\end{equation*}
where $\det(r,x)$ is the determinant of the matrix whose columns
are $r$ and $x$ (in that order).  Equivalently, let $S$ be the
subgroup of skew transformations in $\SL_{2}(\Z)$ that fix $r$,
and let $U_{r}$ be the generator of $S$ that translates a vector
$v$ parallel (resp.\ anti-parallel) to $r$ if the angle between
$r$ and $v$ is less than (resp.\ greater than) $\pi$, measured
counterclockwise about the origin.

Define the piecewise unimodular transformations $U_{r}^{+}$ and
$U_{r}^{-}$ by
\begin{equation*}
   U_{r}^{+}(x)
   =  \begin{dcases*}
         U_{r}(x) & if $\det(r,x) \geq 0$,  \\
         x & else,
      \end{dcases*}
\end{equation*}
and
\begin{equation*}
   U_{r}^{-}(x)
   =
   \begin{dcases*}
      x & if $\det(r, x) \ge 0$, \\
      \smash{(U_{r})^{-1}(x)} & else.
   \end{dcases*}
\end{equation*}
Finally, given a lattice point $u \in \Z^{2}$ and a rational
point $v \in \Q^{2}$, let $U_{uv}^{+}$ and $U_{uv}^{-}$ be the
\emph{affine} piecewise unimodular transformations defined by
\begin{align*}
   U_{uv}^{+}(x)
   &= U_{v - u}^{+}(x - u) + u,  \\
   U_{uv}^{-}(x) 
   &= U_{v - u}^{-}(x - u) + u.
\end{align*}

\section{%
    Constructing nonintegral PIPs%
}%
\label{sec:NonintegralPIPs}

We recall our notation from the introduction.  Given a polygon $P
\subset \R^{2}$, let $\A_{P}$ be the area of $P$, and let
$\intr_{P}$ and $\bnd_{P}$ be the number of lattice points in the
interior and on the boundary of $P$, respectively.  We now prove 
Theorem \ref{thm:BelowScottsRegion}, which we restate here for 
the convenience of the reader.


\begin{thmBelowScottsRegion}
   Given integers $I \ge 1$ and $b \in \{1, 2\}$, there exists a
   polygonal PIP $P$ with $(\intr_{P}, \bnd_{P}) = (I, b)$.
   However, there does not exist a polygonal PIP $P$ with
   $\bnd_{P} = 0$ or with $(\intr_{P}, \bnd_{P}) \in \{ (0,1), (0,
   2) \}$.
\end{thmBelowScottsRegion}

\begin{proof}
   \label{proof:BelowScottsRegion}
   Let integers $b \in \setof{1,2}$ and $I \ge 1$ be given.  We
   construct a polygonal PIP $P$ with $(\intr_{P}, \bnd_{P}) =
   (I,b)$.
   
   If $b = 2$, consider the triangle \[T = \conv
   \setof{(0,0)^{\tran}, (I+1, 0)^{\tran}, (1, 1 -
   \tfrac{1}{I+1})^{\tran}}.\] It was proved in \cite{MW2005}
   that $T$ is a PIP\@.  Let $P$ be the union of $T$ and its
   reflection about the $x$-axis.  Then $\intr_{P} = I$ and
   $\bnd_{P} = 2$.  Moreover, $\ehr_{P}(n) = 2\ehr_{T}(n) - (I +
   2)$ (correcting for points double-counted on the $x$-axis),
   so $P$ is also a PIP.
   
   If $b = 1$, consider the ``semi-open'' triangle 
   \begin{align*}
      T_{1} &= \conv\setof{(0,0)^{\tran}, (1, 2I - 1)^{\tran},
      (-1, 0)} \setminus \bigl( (0,0)^{\tran}, (1, 2I - 1)^{\tran}
      \bigr].
   \end{align*}
   (See Figure \ref{fig:OneVertex-a} for the case with $I = 3$.)
   The Ehrhart quasi-polynomial of $T$ is evidently a signed sum
   of Ehrhart polynomials of integral polytopes, so it too is a
   polynomial.  We will apply a succession of $\pZ$-morphisms
   to $T$ to produce a convex rational polygon without changing
   the Ehrhart polynomial.
   
   \begin{figure}[tbp]
      \def\JPicScale{0.33}
      \hspace*{\fill}%
      \subfloat[]{%
         \includegraphics{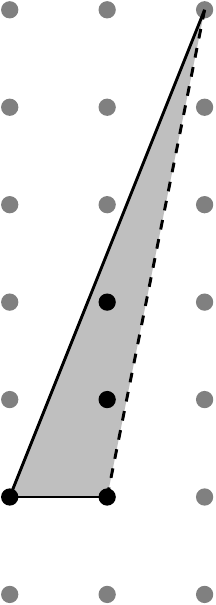}%
         \label{fig:OneVertex-a}%
      }%
      \hfill%
      \subfloat[]{%
         \includegraphics{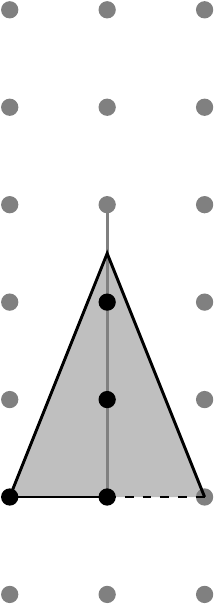}%
         \label{fig:OneVertex-b}%
      }%
      \hfill%
      \subfloat[]{%
         \includegraphics{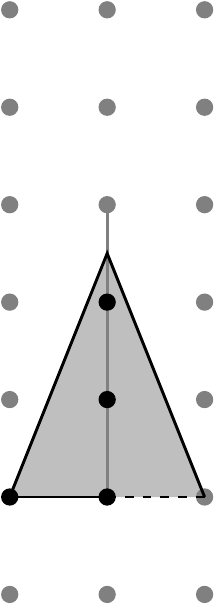}%
         \label{fig:OneVertex-c}%
      }%
      \hspace*{\fill}%
      \caption{%
         The construction of a polygonal PIP with one boundary
         point and an arbitrary number $I$ of interior points in
         the case $I = 3$.  Black points are elements of the
         region.  Gray line segments indicate the lines fixed by
         the skew transformations.
      }%
      \label{fig:OneVertex}
   \end{figure}
   
   Let $T_{2} = (U_{(0,-1)^{\tran}}^{+})^{2I - 1}(T_{1})$.  (See
   Figure \ref{fig:OneVertex-b}.  The gray line segment indicates
   the line fixed by the skew transformation.)  Hence,
   \[T_{2} = \conv \setof{(1,0)^{\tran}, (0, I - 1/2)^{\tran},
   (-1,0)^{\tran}} \setminus \bigl((0,0)^{\tran},
   (1,0)^{\tran}\bigr].\]
   Now act upon the triangle below the line spanned by $(-1,-1)$
   (resp.\ $(1,-1)$), with $U_{(-1,-1)^{\tran}}^{+}$ (resp.\
   $U_{(1,-1)^{\tran}}^{-}$).  (See
   Figure~\ref{fig:OneVertex}\subref*{fig:OneVertex-c}.  The line
   segments meeting at the origin lie on the lines fixed by one of
   these unimodular transformations.)  The result is
   \begin{equation*}
      T_{3} = \conv \left\lbrace
      \begin{pmatrix}
         0  \\
         -1
      \end{pmatrix},\;
      \begin{pmatrix}
         \dfrac{2I - 1}{2I + 1}  \\
         \\
         \dfrac{2I - 1}{2I + 1}
      \end{pmatrix},\;
      \begin{pmatrix}
         0  \\
         I - 1/2
      \end{pmatrix},\;
      \begin{pmatrix}
         -\dfrac{2I - 1}{2I + 1}  \\
         \\
         \dfrac{2I - 1}{2I + 1}
      \end{pmatrix}
    \right \rbrace.
   \end{equation*}
   At this point, we have a
   convex rational polygon with the desired number of interior and
   boundary points, so the claim is proved.  However, it might be
   noted that we can achieve a triangle by letting $P =
   (U_{(0,1)^{\tran}}^{-})^{^{2I - 1}}(T_{3})$, yielding
   \begin{equation*}
      P
      =  \conv
         \setof{
            \begin{pmatrix}
               0  \\
               -1
            \end{pmatrix},\;
            \begin{pmatrix}
               \dfrac{2I - 1}{2I + 1}  \\
               \\
               2I\, \dfrac{2I - 1}{2I + 1}
            \end{pmatrix},\;
            \begin{pmatrix}
               -\dfrac{2I - 1}{2I + 1}  \\
               \\
               \dfrac{2I - 1}{2I + 1}
            \end{pmatrix}
          }. \qedhere
    \end{equation*}
    
    To prove the nonexistence claim, let a polygonal PIP $P$ be
    give.  In \cite[Theorem 3.1]{MW2005}, it was shown that
    $\bnd_{nP} = n \bnd_{P}$ for $n \in \Z_{>0}$.  If $\bnd_{P} =
    0$, this implies that $\bnd_{nP} = 0$ for all $n \in \Z_{>0}$,
    which is impossible because, for example, some integral dilate
    of $P$ is integral.  Hence, $\bnd_{P} \ge 1$.
     
    It was also shown in \cite{MW2005} that polygonal PIPs satisfy
    Pick's theorem: $\A_{P} = \intr_{P} + \frac{1}{2} \bnd_{P} -
    1$.  But if $\intr_{P} = 0$ and $\bnd_{P} \in \{1, 2\}$, this
    yields $\A_{P} \le 0$.  Since our polygons are not contained in
    a line by definition, this is impossible.  Therefore, if
    $\bnd_{P} < 3$, we must have $\bnd_{P} \in \{1, 2\}$ and
    $\intr_{P} \ge 1$.
\end{proof}

A proof of, or counterexample to, Scott's inequality $\bnd_{P} \le
2 \intr_{P} + 7$ for nonintegral polygonal PIPs with interior
points eludes us.  (See the question marks in
Figure~\ref{fig:ScottRegionAll}.)  However, it is easy to show
that any counterexample $P$ cannot contain a lattice point in the
interior of its integral hull $\inhl{P} \deftobe \conv (P \cap
\Z^{2})$.  Indeed, the proof does not even require the hypothesis
that $P$ is a PIP.

\begin{prop}
   If $P$ is a polygon whose integral hull contains a lattice
   point in its interior, then either $(\intr_{P}, \bnd_{P}) = (1,
   9)$ or $\bnd_{P} \le 2 \intr_{P} + 6$.
\end{prop}

\begin{proof}
   We are given that $\intr_{\inhl{P}} \ge 1$.  Note that
   $\bnd_{\inhl{P}} \ge \bnd_{P}$ and $\intr_{\inhl{P}} \le
   \intr_{P}$.  Since $\inhl{P}$ is an integral polygon, it obeys
   Scott's inequality: $\bnd_{\inhl{P}} \le 2 \intr_{\inhl{P}} +
   6$ unless $(\intr_{\inhl{P}}, \bnd_{\inhl{P}}) = (1, 9)$.  In
   the former case, we have
   $%
   \bnd_{P}
   \le \bnd_{\inhl{P}}
   \le 2 \intr_{\inhl{P}} + 6 
   \le 2 \intr_{P} + 6%
   $.  In the latter case, we similarly have $\bnd_{P} \le
   \bnd_{\inhl{P}} = 9$ and $1 = \intr_{\inhl{P}} \le \intr_{P} $,
   so either $\intr_{P} = 1$ or $\bnd_{P} \le 2 \intr_{P} + 6$.
\end{proof}

\section{%
   Periods of coefficients of Ehrhart Quasi-polynomials%
}%
\label{sec:PeriodsOfCoefficients}

If $P$ is a rational polygon, then the coefficient of the leading
term of $\ehr_{P}$ is the area of $P$, so the ``quadratic'' term
in the period sequence of $P$ is always~$1$.  However, we show
below that no constraints apply to the remaining terms in the
period sequence.  This is our Theorem 
\ref{thm:CoefficientPeriodSequences}, which we restate here for 
the convenience of the reader.
\begin{thmCoefficientPeriodSequences}
   Given positive integers $r$ and $s$, there exists a polygon
   $P$ with period sequence $(r, s, 1)$.
\end{thmCoefficientPeriodSequences}
  
Before proceeding to the proof, we make some elementary
observations regarding the coefficients of certain Ehrhart
quasi-polynomials.

Fix a positive integer $s$, and let $\ell$ be the line segment
$[0, \tfrac{1}{s}]$.  Then we have that $\ehr_{\ell}(n) =
\tfrac{1}{s} n + c_{\ell,0}(n)$, where the ``constant''
coefficient function $c_{\ell,0}(n) = \floors{n/s} - n/s + 1$ has
minimum period $s$.  Note also that the half-open interval $h
\deftobe (\tfrac{1}{s}, 1]$ satisfies $\ehr_{\ell} + \ehr_{h} =
\ehr_{[0,1]}$.  In particular, we have that
\begin{equation}
   c_{\ell,0} + c_{h,0} = 1.
   \label{eq:ConstantCoeffofLineSegment}
\end{equation}

Given a positive integer $m$, it is straightforward to compute
that the Ehrhart quasi-polynomial of the rectangle $\ell \times
[0, m]$ is given by
\begin{equation*}
   \ehr_{\ell \times [0, m]}(n)
   =  \tfrac{m}{s}n^{2}
      +  \big(
            m c_{\ell,0} (n)
            +  \tfrac{1}{s}
         \big)
         n
      +  c_{\ell,0} (n).
\end{equation*}
In particular, the ``linear'' coefficient function has minimum
period $s$, and the ``constant'' coefficient function is
identical to that of $\ehr_{\ell}$.  More strongly, we have the
following:
\begin{lem}
\label{lem:CoeffsofRectangleUnionIntegral}
   Suppose that a polygon $P$ is the union of $\ell \times [0,
   m]$ and an integral polygon $P'$ such that $P' \cap (\ell
   \times [0, m])$ is a lattice segment.  Then $c_{P,1}$ has
   minimum period $s$ and $c_{P,0} = c_{\ell,0}$.
\end{lem}

With these elementary facts in hand, we can now prove Theorem
\ref{thm:CoefficientPeriodSequences}.

\begin{proof}[%
   Proof of Theorem \ref{thm:CoefficientPeriodSequences}%
]%
   \label{proof:CoefficientPeriodSequences}
   Any integral polygon has period sequence $(1,1,1)$, so we may
   suppose that either $r \ge 2$ or $s \ge 2$.  Our strategy is to
   construct a polygon $H$ with period sequence $(1, s, 1)$ and a
   triangle $Q$ with period sequence $(r, 1, 1)$.  We will then be
   able to construct a polygon with period sequence $(r, s, 1)$
   for $r,s \ge 2$ by gluing $Q$ to $H$ along an integral edge.
   
   We begin by constructing a polygon with period sequence $(1, s,
   1)$ for an arbitrary integer $s \ge 2$.  (Figure
   \ref{fig:CoefficientPeriodsConstruction} below depicts the $s =
   3$ case of our construction.)  Define $H$ to be the heptagon
   with vertices
      \begin{align*}
         t_{1} 
         &= \big(
               -\tfrac{1}{s},\, s(s-1) + 1
            \big)^{\tran}, &
         v_{1}
         &= \big(
               1,\, s(s-1)
            \big)^{\tran}, \\
         t_{2}
         &= \big(
               -\tfrac{1}{s},\, -s(s-1) - 1
            \big)^{\tran}, &
         v_{2}
         &= \big(
               1,\, -s(s-1)
            \big)^{\tran}, \\
         u_{1}
         &= \big(
               0,\, s(s-1) + 1
            \big)^{\tran}, &
         w 
         &= \big(
               s - 1 + \tfrac{1}{s},\, 0
            \big)^{\tran}. \\
         u_{2}
         &= \big(
               0,\, -s(s-1) - 1
            \big)^{\tran}, &
      \end{align*}%
%
   \begin{figure*}[t!]
      \def\JPicScale{0.33}%
      \hspace*{\stretch{1}}%
      \subfloat[]{%
         \includegraphics{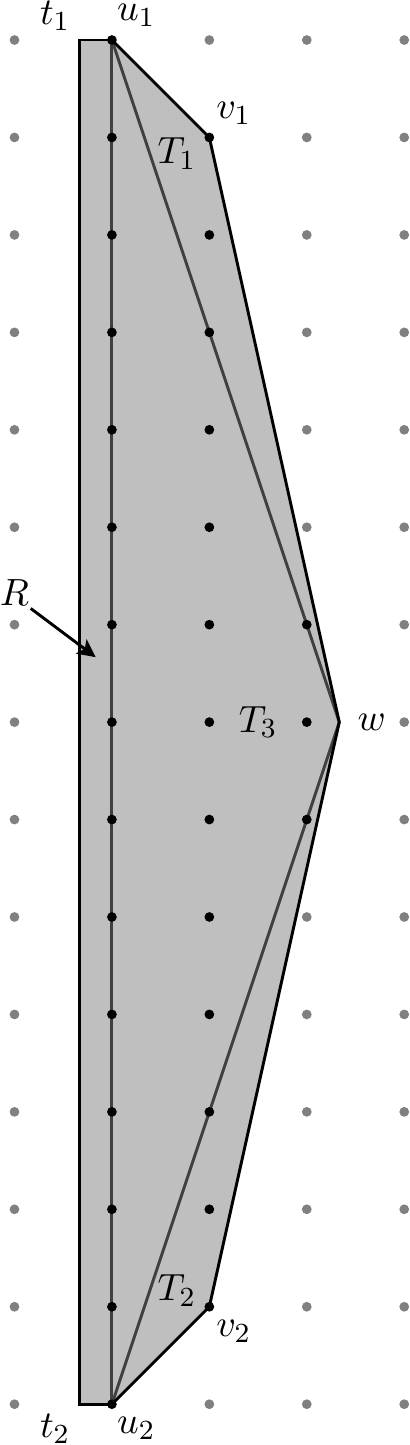}%
      }%
      \hspace*{\stretch{2}}%
      \subfloat[]{%
         \includegraphics{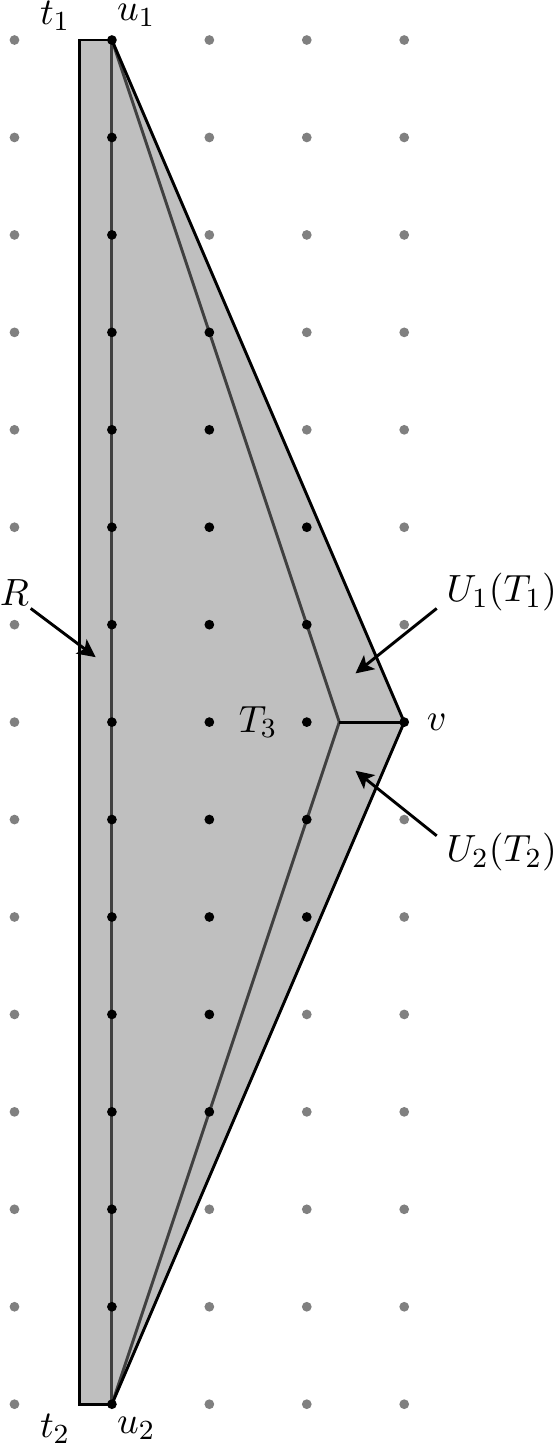}%
      }%
      \hspace*{\stretch{1}}%
      \caption{%
         On left: polygon $H$ in the case $s = 3$.  On right:
         polygon $H'$ after unimodular rearrangement of pieces
         of $H$.
      }%
      \label{fig:CoefficientPeriodsConstruction}%
   \end{figure*}%
   To show that $H$ has period sequence $(1, s, 1)$, we subdivide
   $H$ into a rectangle and three triangles as follows (see left
   of Figure \ref{fig:CoefficientPeriodsConstruction}):
   \begin{align*}
      R     &= \conv\{t_{1}, t_{2}, u_{2}, u_{1}\}, &
      T_{2} &= \conv\{ u_{2}, v_{2}, w \}, \\
      T_{1} &= \conv\{ u_{1}, v_{1}, w \}, &
      T_{3} &= \conv \{ u_{1}, u_{2}, w \}.
   \end{align*}
   Let $v = (s, 0)^{\tran}$.  Write $U_{1} = U_{u_{1} w}^{+}$
   and $U_{2} = U_{u_{2}w}^{-}$.  Then $U_{1}(T_{1}) = \conv
   \setof{u_{1}, v, w}$ and $U_{2}(T_{2}) = \conv \setof{u_{2},
   v, w}$. 
   
   Let $H' = R \cup U_{1}(T_{1}) \cup U_{2}(T_{2}) \cup T_{3}$
   (see right of Figure
   \ref{fig:CoefficientPeriodsConstruction}).
   Though $H'$ was formed from unimodular images of pieces of
   $H$, we do not quite have $\ehr_{H} = \ehr_{H'}$.  This is
   because each point in the half-open segment $(w, v]$ has two
   pre-images in $H$.  Since this segment is equivalent under a
   unimodular transformation to $h = (\tfrac{1}{s}, 1]$, the
   correct equation is
   \begin{equation}
   \label{eq:EhrHprimeFromEhrH}
      \ehr_{H} = \ehr_{H'} + \ehr_{h}.
   \end{equation}
           
   Let \mbox{$T = U_{1}(T_{1}) \cup U_{2}(T_{2}) \cup T_{3}$}.
   Then $T$ is an integral triangle intersecting $R$ along a
   lattice segment, and \mbox{$H' = R \cup T$}.  Hence, by Lemma
   \ref{lem:CoeffsofRectangleUnionIntegral}, $c_{H',1}$ has
   minimum period~$s$, and so, by equation
   \eqref{eq:EhrHprimeFromEhrH}, $c_{H,1}$ also has minimum
   period~$s$.
   
   It remains only to show that $c_{H,0}$ has minimum period $1$.
   Again, from equation \eqref{eq:EhrHprimeFromEhrH}, we have that
   \begin{equation}
   \label{eq:ConstantTerm}
      c_{H,0} = c_{H',0} + c_{h,0}.
   \end{equation}
   From Lemma \ref{lem:CoeffsofRectangleUnionIntegral}, we know
   that $c_{H',0} = c_{\ell,0}$.  Therefore, by equation
   \eqref{eq:ConstantCoeffofLineSegment}, $c_{H,0}$ is
   identically $1$.
   
   
   We now construct a triangle with period sequence $(r, 1, 1)$
   for integral $r \ge 2$.  Let
   \begin{equation*}
      Q = u_{1} + \conv\setof{(0, 0), (1, -1), (1/r, 0)}.
   \end{equation*}
   McMullen's bound (Theorem \ref{thm:McMullenBound}) implies
   that the minimum period of $c_{Q,1}$ is $1$.  Hence, it
   suffices to show that the minimum quasi-period of $\ehr_{Q}$
   is $r$.  Observe that $Q$ is equivalent to
   $\conv\setof{(0,0), (1,0), (0, 1/r)}$ under a unimodular
   transformation.  Hence, one easily computes that
   $\sum_{k=0}^{\infty} \ehr_{Q}(k) \zeta^{k} = (1-\zeta)^{-2}
   (1 - \zeta^{r})^{-1}$.  Note that among the poles of this
   rational generating function are primitive
   $r$\textsuperscript{th} roots of unity.  It follows from the
   standard theory of rational generating functions that
   $\ehr_{Q}$ has minimum quasi-period $r$ (see, \emph{e.g.},
   \cite[Proposition 4.4.1]{Sta1997}).
   
   
   
   Finally, given integers $r, s \ge 2$, let $P = H \cup Q$.  Note
   that $H$ and $Q$ have disjoint interiors, $H \cap Q$ is a
   lattice segment of lattice length $1$, and $H \cup Q$ is
   convex.  It follows that $P$ is a convex polygon and $\ehr_{P}
   = \ehr_{H} + \ehr_{Q} - \ehr_{[0,1]}$.  Therefore, $P$ has
   period sequence $(r, s, 1)$, as required.
\end{proof}

\section{Pseudo-reflexive polygons}

We conclude with some speculative remarks about the connection
between PIPs, reflexive polygons, and $\SL_{2}(\Q)$.  In
particular, nonintegral PIPs that contain only a single lattice
point in their interior appear to be nonintegral analogues of
reflexive polygons.  To explore this connection, we introduce the
notion of a \emph{pseudo-reflexive} polygon.

Recall that an integral polytope $P \subset \R^{n}$ is called
\emph{reflexive} if the polar dual $\dual{P} = \setof{y \in \R^{n}
\st \text{$\inner{x, y} \le 1$ for all $x \in P$}}$ of $P$ is also
an integral polytope.  We similarly define a
\emph{pseudo-reflexive} polytope to be a PIP $P$ such that $\dual
P$ is integral.  An example of a nonintegral pseudo-reflexive PIP
is the convex hull of $\setof{(0, -1),\, (1/3, 1/3),\, (-1/3,
2/3)}$.

Poonen and Rodriguez-Villegas \cite{PooRod2000} showed that, if $P
\subset \R^{2}$ is a reflexive polygon, then $\bnd_{P} +
\bnd_{\dual P} = 12$.  Hille and Skarke \cite{HilSka2002} observed
that this fact follows from a correspondence between reflexive
polygons and words equal to the identity in a certain presentation
of $\SL_{2}(\Z)$.  In particular, put
\begin{equation*}
   A \deftobe 
   \begin{bmatrix}
      1 & 1 \\
      0 & 1
   \end{bmatrix}, \qquad
   B \deftobe
   \begin{bmatrix}
      1 & 0 \\
      -1 & 1
   \end{bmatrix}.
\end{equation*}
Then $\SL_{2}(\Z)$ is generated by $A$ and $B$, and the relations
of this presentation are
\begin{align}
   \label{eq:SLRelations}
   ABA &= BAB, & (AB)^{6} & =
   \begin{bmatrix}
      1 & 0 \\
      0 & 1
   \end{bmatrix}.
\end{align}
Given a word
\begin{equation}
   \label{eq:HSword}
   B^{b_{n}}A^{a_{n}} B^{b_{n-1}}A^{a_{n-1}} \dotsm B^{b_{1}}A^{a_{1}}
   = 
   \begin{bmatrix}
      1 & 0 \\
      0 & 1
   \end{bmatrix}
\end{equation}
in $\SL_{2}(\Z)$ with $a_{i}, b_{i}$ positive integers, one reads
off a closed polygonal path $v_{0}v_{1}\dotsb v_{n-1}v_{0}$ with a
positive winding number $w$ about the origin as follows: Put
$v_{0} \deftobe (1,0)$, $d_{0} \deftobe (0, 1)$, and recursively
define
\begin{equation}
   \label{eq:HSRecursiveRelations}
   \begin{bmatrix}
      v_{i}  \\
      d_{i}
   \end{bmatrix}
   \deftobe
   B^{b_{i}} A^{a_{i}}
   \begin{bmatrix}
      v_{i-1}  \\
   d_{i-1} \end{bmatrix},
   \qquad
   \text{for $1 \le i < n$}.
\end{equation}

If the winding number $w$ equals $1$, then this path is the
boundary of a reflexive polygon $P$ with $\bnd_{P} = \sum_{i}
a_{i}$ and $\bnd_{\dual P} = \sum_{i} {b_{i}}$.
Moreover, every reflexive polygon can be obtained in this way, up
to an automorphism of the lattice.  It follows directly from the
relations \eqref{eq:SLRelations} that $\bnd_{P} + \bnd_{\dual P}$
is a multiple of $12$.  Hille and Skarke show that, in general,
$\sum_{i} a_{i} + \sum_{i} b_{i} = 12 w$, from which the
Poonen--Rodriguez-Villegas result follows.

Our observation is that every pseudo-reflexive polygon $P$ also
corresponds to a word, this time in the infinite set of generators
\begin{equation}
   \label{eq:SLQgenerators}
   A^{r} = 
   \begin{bmatrix}
      1 & r \\
   0 & 1 \end{bmatrix} %
   \, \text{(for all $r \in \Q$)}, %
   \qquad B = 
   \begin{bmatrix}
      1 & 0 \\
      -1 & 1
   \end{bmatrix}
\end{equation}
of $\SL_{2}(\Q)$.  Let $v_{0}, v_{1}, \dotsc, v_{n-1} \in \Q^{2}$
be such that the polygonal path $v_{0}v_{1}\dotsb v_{n-1}v_{0}$ is
the boundary of a pseudo-reflexive PIP in which each $v_{i}$ is a
vertex.  By Theorem \ref{thm:BelowScottsRegion}, this boundary
contains a lattice point.  By applying an automorphism of the
lattice, we may suppose that this lattice point is $v_{0} =
(0,1)$.  Let $a_{i} \in \Q$ be the lattice length of the segment
$v_{i-1}v_{i}$.  Set $b_{i}' \in \Z$ to be the lattice length of
the edge of $\dual P$ with outer normal $v_{i}$.  Put $b_{i}
\deftobe \den(v_{i})b_{i}'$.  Then, corresponding to $P$, we get a
word
\begin{equation*}
   B^{b_{n}}A^{a_{n}} B^{b_{n-1}}A^{a_{n-1}} \dotsm B^{b_{1}}A^{a_{1}}
   = 
   \begin{bmatrix}
      1 & 0 \\
      0 & 1
   \end{bmatrix}
\end{equation*}
in the generators \eqref{eq:SLQgenerators} equal to the identity.

Conversely, given such a word with the $a_{i} \in \Q$, $a_{i} >
0$, $a_{1} + \dotsb + a_{n} \in \Z$, and the $b_{i}$ positive
integers divisible by $\parens{\den\parens{\sum_{j=1}^{i}
a_{j}}}^{2}$, we have a corresponding polygonal path
$v_{0}v_{1}\dotsb v_{n-1}v_{0}$ defined by the same recursive
relations \eqref{eq:HSRecursiveRelations} used by Hille and
Skarke.  If the winding number of this path about the origin is
$1$, then the path is the boundary of a polygonal pseudo-reflexive
polygon.

\bibliographystyle{amsplain-fi-arxlast}
\bibliography{MasterBibliography}
\end{document}